\documentclass[12pt]{amsart}

\usepackage[usenames,dvipsnames]{xcolor}

\usepackage{enumitem,kantlipsum}

\usepackage{mathtools}
\usepackage{xfrac}
\usepackage[utf8]{inputenc}
\usepackage{amssymb}
\usepackage{amsmath}
\usepackage{amsfonts}
\usepackage{amsthm}
\usepackage{tikz}
\usepackage{algorithmic}
\usepackage{asymptote}
\usepackage{url}
\newtheorem{theorem}{Theorem}[section]
\newtheorem{corollary}{Corollary}[theorem]
\newtheorem{lemma}[theorem]{Lemma}
\newtheorem{proposition}[theorem]{Proposition}

\newtheorem{definition}[theorem]{Definition}

\usepackage[OT2,T1]{fontenc}
\DeclareSymbolFont{cyrletters}{OT2}{wncyr}{m}{n}
\DeclareMathSymbol{\Sha}{\mathalpha}{cyrletters}{"58}

\theoremstyle{definition}

\newtheorem{example}[theorem]{Example}

\newtheorem{question}[theorem]{Question}

\newcommand{\Z}{\mathbb{Z}}
\newcommand{\Zbar}{\overline{\mathbb{Z}}}

\newcommand{\D}{\mathfrak{D}}
\newcommand{\OO}{\mathcal{O}}
\newcommand{\p}{\mathfrak{p}}
\newcommand{\q}{\mathfrak{q}}

\newcommand{\LL}{\mathcal{L}}
\newcommand{\Q}{\mathbb{Q}}

\newcommand{\Spec}{\text{Spec}}
\newcommand{\okgm}{\OO_K[
\gamma]}

\newcommand{\Nm}{\text{Nm}}
\newcommand{\Cl}{\text{Cl}}
\newcommand\restr[2]{{% we make the whole thing an ordinary symbol
  \left.\kern-\nulldelimiterspace % automatically resize the bar with \right
  #1 % the function
  \vphantom{\big|} % pretend it's a little taller at normal size
  \right|_{#2} % this is the delimiter
  }}

\newtheorem{innercustomthm}{Theorem}
\newenvironment{customthm}[1]
  {\renewcommand\theinnercustomthm{#1}\innercustomthm}
  {\endinnercustomthm}

\newtheorem{innercustomprop}{Proposition}
\newenvironment{customprop}[1]
  {\renewcommand\theinnercustomprop{#1}\innercustomprop}
  {\endinnercustomprop}

\begin{document}
\title{Primes in denominators of algebraic numbers}
\author{Deepesh Singhal, Yuxin Lin}

\begin{abstract}
Denote the set of algebraic numbers as $\overline{\Q}$ and the set of algebraic integers as $\overline{\Z}$.
For $\gamma\in\overline{\Q}$, consider its irreducible polynomial in $\Z[x]$, $F_{\gamma}(x)=a_nx^n+\dots+a_0$. Denote $e(\gamma)=\gcd(a_{n},a_{n-1},\dots,a_1)$. Drungilas, Dubickas and Jankauskas show in a recent paper that $\Z[\gamma]\cap \Q=\{\alpha\in\Q\mid \{p\mid v_p(\alpha)<0\}\subseteq \{p\mid p|e(\gamma)\}\}$. Given a number field $K$ and $\gamma\in\overline{\Q}$, we show that there is a subset $X(K,\gamma)\subseteq \Spec(\OO_K)$, for which $\OO_K[\gamma]\cap K=\{\alpha\in K\mid \{\p\mid v_{\p}(\alpha)<0\}\subseteq X(K,\gamma)\}$. We prove that $\OO_K[\gamma]\cap K$ is a principal ideal domain if and only if the primes in $X(K,\gamma)$ generate the class group of $\OO_K$.
We show that given $\gamma\in \overline{\Q}$, we can find a finite set $S\subseteq \overline{\Z}$, such that for every number field $K$, we have $X(K,\gamma)=\{\p\in\Spec(\OO_K)\mid \p\cap S\neq \emptyset\}$. We study how this set $S$ relates to the ring $\overline{\Z}[\gamma]$ and the ideal $\D_{\gamma}=\{a\in\overline{\Z}\mid a\gamma\in\overline{\Z}\}$ of $\overline{\Z}$. We also show that $\gamma_1,\gamma_2\in \overline{\Q}$ satisfy $\D_{\gamma_1}=\D_{\gamma_2}$ if and only if $X(K,\gamma_1)=X(K,\gamma_2)$ for all number fields $K$.
\end{abstract}

\maketitle

\section{Introduction}
Let $\overline{\Q}$ be the algebraic closure of $\Q$ and let $\overline{\Z}$ be the set of algebraic integers.
We start with the following question:
\begin{question}
For an algebraic number $\gamma\in\overline{\Q}$, how should we define its denominator?
\end{question}
Denote the minimal polynomial of $\gamma$ in $\Z[x]$ as
$$F_{\gamma}(x)=a_nx^n+a_{n-1}x^{n-1}+\dots+a_0.$$
So $F_{\gamma}(x)\in \Z[x]$ is an irreducible polynomial with $F_{\gamma}(\gamma)=0$, $\gcd(a_n,\dots,a_0)=1$ and $a_n>0$.
We denote $\deg(\gamma)=n$.
The leading coefficient is denoted by $c(\gamma)=a_n$. The smallest positive integer $d$ for which $d\gamma$ is an algebraic integer is denoted as $d(\gamma)$. The numbers $c(\gamma)$ and $d(\gamma)$ are two natural candidates for the denominator of $\gamma$ and $d(\gamma)$ is usually called as the denominator of $\gamma$.
The two are closely related and it is known that
$$d(\gamma)\mid c(\gamma) \mid d(\gamma)^n.$$
In their paper, Arno, Robinson and Wheeler \cite{d(gamma)} compute $d(\gamma)$ in terms of the coefficients of $F_\gamma(x)$ by showing that for every prime $p$
$$v_p(d(\gamma))=\max\left(0,\max_{0\leq j\leq n-1}\left\lceil\frac{v_p(a_n)-v_p(a_j)}{n-j}\right\rceil\right).$$
They also show that $d(\gamma)=c(\gamma)$ if and only if there is no prime $p$ such that $p^2\mid a_n$ and $p\mid a_{n-1}$. Finally, they show that the density of algebraic numbers $\gamma$ for which $c(\gamma)=d(\gamma)$ is $\frac{1}{\zeta(3)}$. Ayad, Bayad and Kihel in \cite{v_p(d)=k}, fix a number field $K$, a rational prime $p$, integer $k\geq 1$ and determine the sets $\{v_{p}(c(\gamma))\mid \gamma\in K, v_p(d(\gamma))=k\}$. They show that this set only depends on $[K:\Q]$, $k$ and how $p$ splits in $K$. There is another related quantity $i(\gamma)=\gcd_{m\in \Z}F_{\gamma}(m)$, that has been studied in several papers \cite{i gamma 1, i gamma 2, i gamma 3, i gamma 4, i gamma 5, i gamma 6}.

The third candidate for the denominator of $\gamma$ should be $$e(\gamma)=\gcd(a_1,a_2,\dots,a_{n}).$$
Note that the $\gcd$ does not include the constant term $a_0$. Drungilas, Dubickas and Jankauskas show the following in \cite{Z[gamma]cap Q}:
\begin{theorem}\cite[Theorem 5]{Z[gamma]cap Q}\label{Thm: Z[gamma]cap Q}
Given an algebraic number $\gamma\in\overline{\Q}$, we have
$$\Z[\gamma]\cap\Q=\left\{\alpha\in\Q\mid \text{for all primes p, if }v_p(\alpha)<0 ,\text{ then } p|e(\gamma)\right\}.$$
\end{theorem}
\begin{corollary}\label{Cor: Z[gamma]cap Q =Z}
An algebraic number $\gamma\in\overline{\Q}$ satisfies $\Z[\gamma]\cap \Q=\Z$ if and only if $e(\gamma)=1$.
\end{corollary}

Note that if $\gamma$ is an algebraic integer, then $e(\gamma)=1$. However, we can have $e(\gamma)=1$ even if $\gamma$ is not an algebraic integer. For example, take $\gamma=\frac{1}{2+i}$, so $F_{\gamma}(x)=5x^2-4x+1$. We have $c(\gamma)=d(\gamma)=5$ and $e(\gamma)=1$, it satisfies $\Z[\frac{1}{2+i}]\cap\Q=\Z$.

In Section \ref{Sec: classify tuples}, we classify which tuples of positive integers occur as $(c(\gamma),d(\gamma),e(\gamma),\deg(\gamma))$.
\begin{theorem}\label{classify c,d,e,n}
Suppose we are given positive integers $c,d,e,n$ with $n\geq 2$. Then there is a algebraic number $\gamma$ for which $c(\gamma)=c$, $d(\gamma)=d$, $e(\gamma)=e$ and $\deg(\gamma)=n$ if and only if $d|c$, $c|d^n$, $e|c$, $c|d^{n-1}e$ and each rational prime $p$ satisfies at least one of the following:
\begin{enumerate}
    \item $v_{p}(d(\gamma))+v_p(e(\gamma))\leq v_p(c(\gamma))$.
    \item $v_p(c(\gamma))= nv_p(d(\gamma))$.
\end{enumerate}
\end{theorem}

For a number field $K$, we denote the ring of algebraic integers in $K$ as $\OO_K$ and the set of non-zero prime ideals of $\OO_K$ as $\Spec(\OO_K)$.
The ring $\Z[\gamma]\cap \Q$ is controlled by the the prime divisors of $e(\gamma)$. We will show that the ring $\OO_K[\gamma]\cap K$ is controlled by certain prime ideals of $\OO_K$ in a similar manner.
Denote the monic irreducible polynomial of $\gamma$ over $K$ as $f_{K,\gamma}(x)=b_nx^n+\dots+b_0\in K[x]$. So, $f_{K,\gamma}(x)$ is irreducible in $K[x]$, $f_{K,\gamma}(\gamma)=0$ and $b_n=1$.

\begin{theorem}\label{X cond equiv}
Given a number ring $\OO_K$, an algebraic number $\gamma$ and a prime $\p\in\Spec(\OO_K)$, the following are equivalent:
\begin{itemize}
    \item For every $i\in [1,n]$: $v_{\p}(b_i)>v_{\p}(b_0)$.
    \item There exists an $ \alpha\in \p$ for which $\frac{1}{\alpha}\in\OO_K[\gamma]$.
    \item For every $\q\in\Spec(\OO_{K(\gamma)})$: if $\q\cap \OO_K=\p$ then $v_{\q}(\gamma)<0$.
\end{itemize}
\end{theorem}

\begin{theorem}\label{Y cond equiv}
Given a number ring $\OO_K$, an algebraic number $\gamma$ and a prime $\p\in\Spec(\OO_K)$, the following are equivalent:
\begin{itemize}
    \item There exists $i\in [0,n-1]$ for which $v_{\p}(b_i)<0$.
    \item There exists $\q\in\Spec(\OO_{K(\gamma)})$ for which $\q\cap \OO_K=\p$ and $v_{\q}(\gamma)<0$.
\end{itemize}
\end{theorem}

The collection of primes in $\Spec(\OO_K)$ that satisfy the conditions of Theorem~\ref{X cond equiv} is denoted by $X(X,\gamma)$. The collection of primes that satisfy the conditions of Theorem~\ref{Y cond equiv} is denoted by $Y(K,\gamma)$. Clearly we have $X(K,\gamma)\subseteq Y(K,\gamma)$.
Note that
$$X(\Q,\gamma)=\{p\in\Spec(\Z)\mid p|e(\gamma)\},$$
and
$$Y(\Q,\gamma)=\{p\in\Spec(\Z)\mid p|c(\gamma)\}=\{p\in\Spec(\Z)\mid p|d(\gamma)\}.$$
We generalize Theorem \ref{Thm: Z[gamma]cap Q} and Corollary \ref{Cor: Z[gamma]cap Q =Z} to an arbitrary number ring $\OO_K$. The set $X(K,\gamma)$ acts as the analog of prime divisors of $e(\gamma)$. 
\begin{theorem}\label{X=empty result}
Given a number ring $\OO_K$ and an algebraic number $\gamma$, the following hold:
\begin{enumerate}
    \item $X(K,\gamma)=\emptyset$ if and only if $\OO_K[\gamma]\cap K=\OO_K$.
    \item $Y(K,\gamma)=\emptyset$ if and only if $\gamma\in\overline{\Z}$.
\end{enumerate}
\end{theorem}
\begin{proposition}\label{Determine OK[gamma]cap K}
For $\alpha\in K$, we have $\alpha\in \OO_K[\gamma]$ if and only if
$$\{\p\in\Spec(\OO_K)\mid v_{\p}(\alpha)<0\}\subseteq X(K,\gamma).$$
\end{proposition}
Moreover, the primes in $X(K,\gamma)$ also control the the class group of $\OO_K[\gamma]\cap K$.
\begin{proposition}\label{Class group OK[gamma]cap K}
If $K$ is a number field and $\gamma\in\overline{\Q}$, then $\OO_K[\gamma]\cap K$ is a Dedekind domain and its class group is
$$\Cl(\OO_{K}[\gamma]\cap K)\cong \Cl(\OO_K)/\langle[\p]\mid \p\in X(K,\gamma)\rangle.$$
\end{proposition}
The following corollaries follow at once from Proposition \ref{Class group OK[gamma]cap K}.
\begin{corollary}
If $K$ is a number field and $\gamma\in\overline{\Q}$, then $\Cl(\OO_{K}[\gamma]\cap K)\cong \Cl(\OO_K)$ if and only if all prime ideals in $X(K,\gamma)$ are principal.
\end{corollary}

\begin{corollary}
If $K$ is a number field and $\gamma\in\overline{\Q}$, then $\OO_{K}[\gamma]\cap K$ is a principal ideal domain if and only if $\Cl(\OO_K)=\langle[\p]\mid \p\in X(K,\gamma)\rangle$.
\end{corollary}

Given a fixed algebraic number $\gamma\in\overline{\Q}$, we consider the problem of determining $X(K,\gamma)$ for every number field $K$. Firstly, if $\gamma\in K$, then Theorem \ref{X cond equiv} implies that $\p\in X(K,\gamma)$ if and only if $v_{\p}(\gamma)<0$. However, if $\gamma\notin K$, then we need to look at the primes of $\OO_{K(\gamma)}$ that are above $\p$. There is a simpler way to test whether $\p\in X(K,\gamma)$.
%In Section \ref{Sec Gen set} we define $\LL(\gamma)$, it is a collection of certain subfields of $\Q(\gamma)$.

\begin{definition}
Let $S\subseteq \overline{\Z}$, then $S$ is called a generating set of $\gamma$ if it has the property that for every number field $L$,
$$X(L,\gamma)=\{\q\in\Spec(\OO_L)\mid \q\cap S\neq\emptyset\}.$$
\end{definition}

For example, we will show that $\{15,4+i\}$ is a generating set of $\frac{1}{60+15i}$. This means that for any number field $K$, we have $$X\Big(K,\frac{1}{60+15i}\Big)=\{\p\in\Spec(\OO_K)\mid 15\in \p \text{ or }4+i\in \p\}.$$
In Section \ref{Sec Gen set}, we will show that every algebraic number $\gamma$ has a generating set of finite size.
Proposition \ref{Determine OK[gamma]cap K} describes the ring $\OO_K[\gamma]\cap K$ in terms of $X(K,\gamma)$. We can also describe this ring in terms of a generating set of $\gamma$.
\begin{proposition}\label{OK[gamma] in terms of S}
For a subset $S\subseteq\Zbar$,
$S$ is a generating set of $\gamma$ if and only if for every number field $K$, we have
$$\OO_K[\gamma]\cap K=\OO_K\big[1/\alpha\mid \alpha\in S\cap K\big].$$
\end{proposition}
\begin{proposition}\label{Zbar[gamma in terms of S]}
Suppose $S$ is a generating set of $\gamma$. Then we have
$$\Zbar[\gamma]=\Zbar\big[1/\alpha\mid \alpha\in S\big].$$
\end{proposition}
Given an algebraic number $\gamma$, we define its denominator ideal to be $$\D_{\gamma}=\{a\in \overline{\Z}\mid a\gamma\in \overline{\Z}\}.$$
Note that $\D_{\gamma}\cap \Z= d(\gamma)\Z$.
\begin{theorem}\label{radical Dgamma 3 results}
Let $S$ be a finite generating set of $\gamma$, then as ideals of $\overline{\Z}$, we have
$$\sqrt{\D_{\gamma}}=\sqrt{\Big(\prod_{\alpha\in S}\alpha\Big)\overline{\Z}}.$$
Let $K$ be a number field, then as ideals of $\OO_K$ we have
$$\sqrt{\D_{\gamma}\cap \OO_K}=\prod_{\p\in Y(K,\gamma)}\p$$
and
$$\sqrt{\Big(\prod_{\alpha\in S\cap K}\alpha\Big)\OO_K}=\prod_{\p\in X(K,\gamma)}\p.$$
\end{theorem}

In Theorem \ref{X(gamma1=Xgamma2)}, we show that two algebraic numbers $\gamma_1,\gamma_2$ satisfy $\D_{\gamma_1}=\D_{\gamma_2}$ if and only if for any number field $K$, $X(K,\gamma_1)=X(K,\gamma_2)$.
In Proposition \ref{local case}, we obtain an analog of Proposition \ref{Determine OK[gamma]cap K} for finite extensions of $\Q_p$.

\section{Primes in denominators}
In this section we will prove Theorem \ref{X cond equiv}, Theorem \ref{Y cond equiv}, Theorem \ref{X=empty result}, Proposition \ref{Determine OK[gamma]cap K} and Proposition \ref{Class group OK[gamma]cap K}.
Let $K$ be a number field and $\gamma\in\overline{\Q}$ be an algebraic number. Recall that the monic irreducible polynomial of $\gamma$ in $K[x]$ is denoted as $f_{K,\gamma}(x)=b_nx^n+\dots+b_0\in K[x]$.
We have $b_n=1$. For each non-zero prime ideal $\p\in\Spec(\OO_K)$, let
$$\alpha_{\p}=-\min\{v_{\p}(b_i)\mid 0\leq i\leq n\}.$$
Clearly $\alpha_{\p}=0$ for all but finitely many primes $\p$. Since $b_n=1$, we also know that all $\alpha_{\p}\geq 0$. Consider the ideal
$I_{K,\gamma}=\prod_{\p\in\Spec(\OO_K)}\p^{\alpha_{\p}}$.

\begin{proposition}\label{OK[gamma]}
Let $K$, $\gamma$ and $I_{K,\gamma}$ be as above, then we have
$$\OO_K[\gamma]\cong \OO_K[x]/(f_{K,\gamma}(x)I_{K,\gamma}[x]).$$
\end{proposition}
\begin{proof}
First we check that $f_{K,\gamma}(x)I_{K,\gamma}[x]\subseteq \OO_K[x]$. For any $b\in I_{K,\gamma}$ and $\p\in\Spec(\OO_K)$, we know that $\p^{\alpha_{\p}}| I_{K,\gamma}$ and $I_{K,\gamma}|(b)$, so $v_{\p}(b)\geq \alpha_{\p}$. By the definition of $\alpha_{\p}$, we see that for each $0\leq i\leq n$, $v_{\p}(bb_i)\geq 0$. Since this holds for all $\p\in\Spec(\OO_K)$, we have $bb_i\in \OO_K$. This means that for any $b\in I_{K,\gamma}$, we have $bf_{K,\gamma}(x)\in \OO_K[x]$ and hence $f_{K,\gamma}(x)I_{K,\gamma}[x]\subseteq \OO_K[x]$.

Now consider the homomorphism $\phi$ from $\OO_K[x]$ to $\OO_K[\gamma]$ that sends $x$ to $\gamma$. It is clearly surjective. It is also clear that $f_{K,\gamma}(x)I_{K,\gamma}[x]\subseteq \ker(\phi)$. It only remains to show that $\ker(\phi)\subseteq f_{K,\gamma}(x)I_{K,\gamma}[x]$.

Pick some $g(x)\in\ker(\phi)$, so $g(x)\in \OO_K[x]$. Since $g(\gamma)=0$, we know that $f_{K,\gamma}|g$ in $K[x]$. So we have $h(x)\in K[x]$, such that $g(x)=f_{K,\gamma}(x)h(x)$. Say
$h(x)=d_mx^m+\dots+d_0$ and $g(x)=c_{m+n}x^{m+n}+\dots +c_0$. So $d_i\in K$ and $c_i\in \OO_K$. For $\p\in\Spec(\OO_K)$, denote
$\beta_{\p}=\min\{v_{\p}(d_j)\mid 0\leq j\leq m\}$.

Fix a prime $\p\in\Spec(\OO_K)$. Let $t$ be the largest integer for which $v_{\p}(b_t)=-\alpha_{\p}$. Also let $k$ be the largest integer for which $v_{\p}(d_k)=\beta_{\p}$. Then we have
$$c_{t+k}=\sum_{i>0}b_{t+i}d_{k-i}+b_td_k+\sum_{j>0}b_{t-j}d_{k+j}.$$
Now, notice that $v_{\p}(b_td_k)=\beta_{\p}-\alpha_{\p}$. For $i>0$, we have $v_{\p}(b_{t+i})>-\alpha_{\p}$ and $v_{\p}(d_{k-i})\geq \beta_{\p}$ meaning $v_{\p}(b_{t+i}d_{k-i})>\beta_{\p}-\alpha_{\p}$. Similarly for $j>0$ we have $v_{\p}(b_{t-j}d_{k+j})>\beta_{\p}-\alpha_{\p}$. We therefore see that
$v_{\p}(c_{t+k})=\beta_{\p}-\alpha_{\p}$.
Next since $c_{t+k}\in \OO_K$, we conclude $\alpha_{\p}\leq \beta_{\p}$. This means that for each $0\leq j\leq m$, we have
$v_{\p}(d_j)\geq \alpha_{\p}=v_{\p}(I_{K,\gamma})$.

Since this holds for each prime $\p\in\Spec(\OO_K)$, we conclude that $I_{K,\gamma}|(d_j)$, that is $d_j\in I_{K,\gamma}$. This implies that $h(x)\in I_{K,\gamma}[x]$ and $g(x)\in f_{K,\gamma}(x)I_{K,\gamma}[x]$.
We have shown that $\ker(\phi)=f_{K,\gamma}(x)I[x]$ and hence we conclude that $\OO_K[\gamma]\cong \OO_K[x]/(f_{K,\gamma}(x)I_{K,\gamma}[x])$.
\end{proof}

Let $K$, $\gamma$ and $f_{K,\gamma}(x)=b_nx^n+\dots+b_0$ be as above ($b_n=1$). Denote
$$S_{K,1}(\gamma)=\Big\{\p\in\Spec(\OO_K)\mid \text{there exists } \alpha\in \p \text{ such that } 1/\alpha\in\OO_K[\gamma]\Big\},$$
$$S_{K,2}(\gamma)=\{\p\in\Spec(\OO_K)\mid \text{for every } i\in [1,n], \text{ we have } v_{\p}(b_i)>v_{\p}(b_0)\},$$
$$S_{K,3}(\gamma)=\{\p\in\Spec(\OO_{K})\mid \text{for every } \q\in\Spec(\OO_{K(\gamma)}),\text{ if } \q\cap \OO_K=\p\text{ then } v_{\q}(\gamma)<0\}.$$
Recall that Theorem \ref{X cond equiv} asks us to show that $S_{K,1}(\gamma)=S_{K,2}(\gamma)=S_{K,3}(\gamma)$. We will do this by showing that $S_{K,1}(\gamma)\supseteq S_{K,2}(\gamma)$, $S_{K,2}(\gamma)\supseteq S_{K,3}(\gamma)$ and $S_{K,3}(\gamma)\supseteq S_{K,1}(\gamma)$.

\begin{lemma}\label{2 in 1}
Given number field $K$ and $\gamma\in\overline{\Q}$, we have
$S_{K,2}(\gamma)\subseteq S_{K,1}(\gamma)$.
\end{lemma}
\begin{proof}
Suppose $\p \in S_{K,2}(\gamma)$. Pick $a\in\OO_K$ such that each $ab_i\in \OO_K$. Denote $c_i=ab_i\in\OO_K$. Then we have $v_\p(c_n),\dots,v_\p(c_1)>v_\p(c_0)$ and  
$$c_0=-c_n\gamma^n-c_{n-1}\gamma^{n-1}-\dots-c_1\gamma.$$
Suppose $(c_0)=\p^e J$, where $e=v_{\p}(c_0)$ and $J+\p=(1)$. Let $h$ be the class number of $K$. Suppose $\p^h=(\beta)$ and $J^h=(\alpha)$. Therefore $c_0^h=\beta^e\alpha$ (possibly after replacing $\alpha$ by a unit multiple) and $(\alpha)+(\beta)=(1)$. Moreover, $\beta\in\p$, $v_{\p}(\beta)=h$ and $v_{\p'}(\beta)=0$ for any $\p'\neq\p$.
Now, we have
$$\beta^e\alpha=c_0^h=(-c_n\gamma^n-c_{n-1}\gamma^{n-1}-\dots-c_1\gamma)^h=(-1)^h\sum_{1\leq j_1\leq\dots\leq j_h\leq n} c_{j_1}\dots c_{j_h} \gamma^{\sum_{i=1}^{h}j_i}.$$
Notice that $v_{\p}(c_{j_1}\dots c_{j_h})\geq h(e+1)=v_{\p}(\beta^{e+1})$. This means that $v_{\p}\left(\frac{c_{j_1}\dots c_{j_h}}{\beta^{e+1}}\right)\geq 0$. Since $v_{\p'}(\beta)=0$ for other primes $\p'$, we conclude that $\frac{c_{j_1}\dots c_{j_h}}{\beta^{e+1}}\in\OO_K$. Therefore,
$$\frac{\alpha}{\beta}=(-1)^h\sum_{1\leq j_1\leq\dots\leq j_h\leq n} \frac{c_{j_1}\dots c_{j_h}}{\beta^{e+1}} \gamma^{\sum_{i=1}^{h}j_i} \in \OO_K[\gamma].$$
Now since $(\alpha)+(\beta)=(1)$, there are $x, y \in \OO_K$ such that $x\alpha+y\beta=1$. Thus, $\frac{1}{\beta}=x\frac{\alpha}{\beta}+y\in\okgm$. So $\beta \in \p$ and $\frac{1}{\beta} \in \okgm$. Thus $\p$ in $S_{K,1}(\gamma)$.
\end{proof}

\begin{lemma}\label{1 in 3}
Given a number field $K$ and $\gamma\in\overline{\Q}$, we have
$S_{K,1}(\gamma)\subseteq S_{K,3}(\gamma)$.
\end{lemma}
\begin{proof}
Suppose we have a prime $\p$ in $S_{K,1}(\gamma)$. This means that we have $\alpha\in\p$ and $c_k\in\OO_K$ such that
$\frac{1}{\alpha}=\sum_{k=0}^{m}c_k\gamma^k$. Let $\q$ be a prime of $\OO_{K[\gamma]}$ above $\p$. Since $\alpha\in\p\subseteq \q$, we see that $v_{\q}\left(\sum_{k=0}^{m}c_k\gamma^k\right)<0$. This implies there is some $k$ for which $v_{\q}(c_k\gamma^k)<0$. Since $c_k\in\OO_K$, we conclude that $v_{\q}(\gamma)<0$.
\end{proof}

Let $\OO_{K\p}$ be the completion of $\OO_K$ at $\p$.
Let $\overline{\OO_{K\p}}$ be an algebraic closure of $\OO_{K\p}$. Let $\overline{v}$ be the unique valutaion on $\overline{\OO_{K\p}}$ that extends $v_{\p}$. Say we have
$$f_{K,\gamma}(x)=(x-\gamma_1)\dots(x-\gamma_n)$$
over $\overline{\OO_{K\p}}$.
Suppose we have the following factorisation in $\OO_{K(\gamma)}$
$$\p\OO_{K(\gamma)}=\q_1^{e_1}\dots \q_r^{e_r}.$$
It is known that
$$\{e_iv_{\q_i}(\gamma)\mid i\in \{1,\dots,r\}\}=\left\{\overline{v}(\gamma_j)\mid j\in\{1,\dots,n\}\right\}.$$

\begin{lemma}\label{3 in 2}
Given a number field $K$ and $\gamma\in\overline{\Q}$, we have
$S_{K,3}(\gamma)\subseteq S_{K,2}(\gamma)$.
\end{lemma}
\begin{proof}
Suppose we have $\p\in S_{K,3}(\gamma)$. This means that all $v_{\q_i}(\gamma)<0$ and hence all $\overline{v}(\gamma_j)<0$. Now, we have
$$\overline{v}\left(b_0\right)=\sum_{j=1}^n\overline{v}(\gamma_j)<0=\overline{v}(b_n).$$
Also for $1\leq t\leq n-1$ we have
$$\overline{v}\left(b_{n-t}\right)\geq \min_{1\leq j_1<\dots<j_t\leq n} \sum_{k=1}^{t}\overline{v}(\gamma_{j_k}) >\sum_{j=1}^n\overline{v}(\gamma_j)=\overline{v}\left(b_0\right).$$
We conclude that $\p\in S_{K,2}(\gamma)$.
\end{proof}

\begin{customthm}{\ref{X cond equiv}}
Given a number ring $\OO_K$, an algebraic number $\gamma$ and a prime $\p\in\Spec(\OO_K)$, the following are equivalent:
\begin{itemize}
    \item For every $i\in [1,n]$: $v_{\p}(b_i)>v_{\p}(b_0)$.
    \item There exists an $ \alpha\in \p$ for which $\frac{1}{\alpha}\in\OO_K[\gamma]$.
    \item For every $\q\in\Spec(\OO_{K(\gamma)})$: if $\q\cap \OO_K=\p$ then $v_{\q}(\gamma)<0$.
\end{itemize}
\end{customthm}
\begin{proof}
This follows from Lemma \ref{2 in 1}, Lemma \ref{1 in 3} and Lemma \ref{3 in 2}.
\end{proof}

\begin{customthm}{\ref{Y cond equiv}}
Given a number ring $\OO_K$, an algebraic number $\gamma$ and a prime $\p\in\Spec(\OO_K)$, the following are equivalent:
\begin{itemize}
    \item There exists $i\in [0,n-1]$ for which $v_{\p}(b_i)<0$.
    \item There exists $\q\in\Spec(\OO_{K(\gamma)})$ for which $\q\cap \OO_K=\p$ and $v_{\q}(\gamma)<0$.
\end{itemize}
\end{customthm}
\begin{proof}
Suppose $\p$ does not satisfy the second condition, this means that each $v_{\q_i}(\gamma)\geq 0$ and hence each $\overline{v}(\gamma_j)\geq 0$. We see that for $1\leq t\leq n$,
$$\overline{v}\left(b_{n-t}\right) \geq\min_{1\leq j_1<\dots<j_t\leq n} \sum_{k=1}^{t}\overline{v}(\gamma_{j_k})
\geq 0.$$
Therefore $\p$ does not satisfy the first condition.

For the other direction, suppose $\p$ satisfies the second condition. This means that at least one $v_{\q_i}(\gamma)< 0$ and hence at least one $\overline{v}(\gamma_j)< 0$. Relabel the $\gamma_j$ such that $\overline{v}(\gamma_1),\dots,\overline{v}(\gamma_t)$ are negative ($t\geq 1$) and $\overline{v}(\gamma_{t+1}),\dots,\overline{v}(\gamma_n)$ are not negative. Then
$$\min\Big\{\sum_{k=1}^{t}\overline{v}(\gamma_{j_k})\mid 1\leq j_1<\dots<j_t\leq n\Big\}
=\sum_{k=1}^{t}\overline{v}(\gamma_{k}),$$
and it is the unique minimum element in the set. Therefore
$$\overline{v}\left(b_{n-t}\right)=\sum_{k=1}^{t}\overline{v}(\gamma_{k})<0.$$
This shows that $\p$ satisfies the first condition.
\end{proof}

\begin{customprop}{\ref{Determine OK[gamma]cap K}}
For $a,b\in \OO_K$, $b\neq 0$: we have $\frac{a}{b}\in \OO_K[\gamma]$ if and only if
$$\{\p\in\Spec(\OO_K)\mid v_{\p}(a)<v_{\p}(b)\}\subseteq X(K,\gamma).$$
\end{customprop}
\begin{proof}
Suppose we have $a, b\in \OO_K$, $b\neq 0$ with $\frac{a}{b} \in \okgm$. Consider some $\p\in\Spec(\OO_K)$ such that $v_\p(a)<v_\p(b)$. Suppose $(a)=\p^e I$ and $(b)=\p^k J$, where $e=v_{\p}(a)$, $k=v_{\p}(b)$ and $I+\p=(1)$. Let $h$ be the class number of $\OO_K$, so $\p^h$, $I^h$ and $J^h$ are all principal ideals. Say $\p^h=(\beta)$, $I^h=(\alpha_1)$ and $J^h=(\alpha_2)$. Therefore (after replacing $\alpha_1$ by a unit multiple) we have $\frac{a^h}{b^h}=\frac{\alpha_1}{\beta^r\alpha_2}\in\OO_K[\gamma]$, where $r=k-e \geq 1$. Since $\alpha_2 \in \OO_K$, we have that $\frac{\alpha_1}{\beta^r}\in \okgm$. Next $I+\p=(1)$ implies $(\alpha_1)+(\beta^r)=1$. This means there exist $x,y \in \OO_K$ such that $x\beta^r+y\alpha_1=1$.
So then $\frac{1}{\beta^r}=x+y\frac{\alpha_1}{\beta^r}\in \okgm$. Finally, $\beta^r \in \p$ and $\frac{1}{\beta^r}\in \okgm$, therefore $\p \in X(K,\gamma)$.

Conversely, suppose we have $a,b\in\OO_K$ such that $b\neq 0$ and
$$\{\p\in\Spec(\OO_K)\mid v_{\p}(a)<v_{\p}(b)\}\subseteq X(K,\gamma).$$
Let $\p_1, \p_2, \dots, \p_n$ be the list of primes such that $v_{\p_i}(b) >v_{\p_i}(a)$. Then since $\p_i \in X(K,\gamma),$ there are $\alpha_i \in \p_i $ such that $\frac{1}{\alpha_i} \in \okgm$.
We can pick a sufficiently large $N$ such that $v_\p(a\prod_i\alpha_i^{N}) \geq v_\p(b)$ for all $\p\in\Spec(\OO_K)$. Then $\frac{a\prod_i\alpha_i^{N}}{b} \in \OO_K$. Finally this implies
\[\frac{a}{b}=\frac{a\prod_i\alpha_i^{N}}{b}\frac{1}{\alpha_1^N}\dots\frac{1}{\alpha_n^N}\in\okgm.\qedhere\]
\end{proof}

\begin{customthm}{\ref{X=empty result}}
Given a number ring $\OO_K$ and an algebraic number $\gamma$:
\begin{enumerate}
    \item $X(K,\gamma)=\emptyset$ if and only if $\OO_K[\gamma]\cap K=\OO_K$.
    \item $Y(K,\gamma)=\emptyset$ if and only if $\gamma\in\overline{\Z}$.
\end{enumerate}
\end{customthm}
\begin{proof}
Note that $X(K,\gamma)=S_{K,1}(\gamma)$. Let $h$ be the class number of $\OO_K$.

First suppose $X(K,\gamma)=\emptyset$. Consider $\alpha\in\OO_K[\gamma]\cap K$. By Proposition \ref{Determine OK[gamma]cap K}, we know that
$$\{\p\in\Spec(\OO_K)\mid v_{\p}(\alpha)<0\}\subseteq X(K,\gamma).$$
This implies that $v_{\p}(\alpha)\geq0$ for all primes $\p\in\Spec(\OO_K)$. Thus $\alpha\in\OO_K$ and hence $\OO_K[\gamma]\cap K=\OO_K$.

Conversely, suppose $X(K,\gamma)\neq\emptyset$. Pick $\p \in X(K,\gamma)=S_{K,1}(\gamma)$. By the definition of $S_{K,1}(\gamma)$, there exists $ \alpha \in \p$ such that $\frac{1}{\alpha} \in \OO_K[\gamma]$. Then $\frac{1}{\alpha}$ is an element in $\OO_K[\gamma]\cap K$ that is not in $\OO_K$. We conclude that $\OO_K[\gamma]\cap K\neq \OO_K$.

Next, we prove the second part. Note that $\gamma\in\overline{\Z}$ if and only if $\gamma$ is integral over $\OO_K$. Moreover, this happens if and only if $f_{K,\gamma}(x)\in\OO_K[x]$. This is equivalent to the condition that for every $\p\in\Spec(\OO_K)$ and for every $i\in[0,n-1]$: $v_{\p}(b_i)\geq 0$. This is clearly equivalent to $Y(K,\gamma)=\emptyset$.
\end{proof}

\begin{proposition}\cite[Chapter 8, Section 3]{Neukirch 2}\label{class group}
Suppose $K$ is a number field and $X$ is a finite set of prime ideals of $\OO_K$. Consider
$$\OO_{K,X}=\{\alpha\in K\mid \text{for all } \p\in\Spec(\OO_K), \text{ if }v_{\p}(\alpha)<0,\text{ then }\p\in X\}.$$
Then $\OO_{K,X}$ is a Dedekind domain and its class group is
$$\Cl(\OO_{K,X})\cong \Cl(\OO_K)/\langle[\p]\mid \p\in X\rangle.$$
\end{proposition}

\begin{customprop}{\ref{Class group OK[gamma]cap K}}
If $K$ is a number field and $\gamma\in\overline{\Q}$, then $\OO_K[\gamma]\cap K$ is a Dedekind domain and its class group is
$$\Cl(\OO_{K}[\gamma]\cap K)\cong \Cl(\OO_K)/\langle[\p]\mid \p\in X(K,\gamma)\rangle.$$
\end{customprop}
\begin{proof}
This follows from Proposition \ref{Determine OK[gamma]cap K} and Proposition \ref{class group}.
\end{proof}

\section{Varying the number field}\label{var num field}

In this Section we will study how $X(K,\gamma)$ and $Y(K,\gamma)$ change as we fix $\gamma$ and change $K$.

\begin{lemma}\label{Going up and down}
If $K\subseteq L$ are number fields then
$$X(K,\gamma)=\{\p\in\Spec(\OO_K)\mid \text{for every }\q\in\Spec(\OO_L), \text{ if } \q\cap\OO_K=\p \text{ then } \q\in X(L,\gamma)\},$$
$$X(L,\gamma)\supseteq\{\q\in\Spec(\OO_L)\mid \q\cap\OO_K\in X(K,\gamma)\}.$$
Moreover if $L\cap K(\gamma)=K$, then we have
$$X(L,\gamma)=\{\q\in\Spec(\OO_L)\mid \q\cap\OO_K\in X(K,\gamma)\}.$$
\end{lemma}
\begin{proof}
First, suppose $\p$ is a prime of $\OO_K$, such that every prime of $\OO_L$ above $\p$ lies in $X(L,\gamma)$. We will show that $\p\in X(K,\gamma)$.
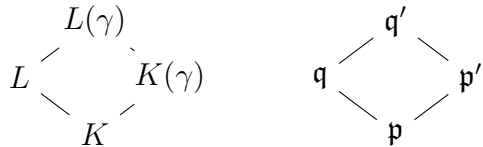
\begin{center}
\begin{tikzpicture}
\node (1) at (0,0) {$K$};
\node (2) at (-1,0.75) {$L$};
\node (3) at (1,0.75) {$K(\gamma)$};
\node (4) at (0,1.5) {$L(\gamma)$};
\node (5) at (4,0) {$\p$};
\node (6) at (3,0.75) {$\q$};
\node (7) at (5,0.75) {$\p'$};
\node (8) at (4,1.5) {$\q'$};
\draw [black] (1) -- (2);
\draw [black] (1) -- (3);
\draw [black] (4) -- (2);
\draw [black] (4) -- (3);
\draw [black] (5) -- (6);
\draw [black] (5) -- (7);
\draw [black] (8) -- (6);
\draw [black] (8) -- (7);
\end{tikzpicture}    
\end{center}
Let $\p'$ be a prime of $\OO_{K(\gamma)}$ that is above $\p$. Let $\q'$ be a prime of $\OO_{L(\gamma)}$ that is above $\p'$. Let $\q=\q'\cap \OO_L$. Then $\q$ lies above $\p$ and hence $\q\in X(L,\gamma)$. This implies that $v_{\q'}(\gamma)<0$. Hence $v_{\p'}(\gamma)=\frac{1}{e(\q'|\p')}v_{\q'}(\gamma)<0$. Since this holds for every prime $\p'$ of $\OO_{K(\gamma)}$ that is above $\p$, we see that $\p\in X(K,\gamma)$. We have shown that
$$X(K,\gamma)\supseteq\{\p\in\Spec(\OO_K)\mid \text{for every }\q\in\Spec(\OO_L), \text{ if } \q\cap\OO_K=\p \text{ then } \q\in X(L,\gamma)\}.$$

Next, consider a prime $\p\in X(K,\gamma)$. Let $\q$ be a prime of $\OO_L$ above $\p$. We will show that $\q\in X(L,\gamma)$. Consider a prime $\q'$ of $\OO_{L(\gamma)}$ above $\q$. Let $\p'=\q'\cap \OO_{K(\gamma)}$. Then $\p$ is in $X(K,\gamma)$ and $\p'$ is a prime of $\OO_{K(\gamma)}$ above $\p$, therefore $v_{\p'}(\gamma)<0$. Next $v_{\q'}(\gamma)=e(\q'|\p')v_{\p'}(\gamma)<0$. Since this holds for every prime $\q'$ of $\OO_{L(\gamma)}$ above $\q$, we conclude that $\q\in X(L,\gamma)$. This shows that
$$X(K,\gamma)\subseteq\{\p\in\Spec(\OO_K)\mid \text{for every }\q\in\Spec(\OO_L), \text{ if } \q\cap\OO_K=\p \text{ then } \q\in X(L,\gamma)\}.$$
We can conclude that
$$X(K,\gamma)=\{\p\in\Spec(\OO_K)\mid \text{for every }\q\in\Spec(\OO_L), \text{ if } \q\cap\OO_K=\p \text{ then } \q\in X(L,\gamma)\}.$$

Next, let $\q$ be a prime of $\OO_L$ for which $\q\cap\OO_K\in X(K,\gamma)$. Let $\p=\q\cap \OO_K$. Then $\p$ is in $X(K,\gamma)$ and $\q$ is a prime of $\OO_L$ above $\p$. Therefore, from the previous part we see that $\q\in X(L,\gamma)$. This shows that
$$X(L,\gamma)\supseteq\{\q\in\Spec(\OO_L)\mid \q\cap\OO_K\in X(K,\gamma)\}.$$

For the last part, suppose $L\cap K(\gamma)=K$. This means that $$f_{L,\gamma}(x)=f_{K,\gamma}(x)=b_nx^n+\dots+b_0\in K(x).$$
Consider $\q\in X(L,\gamma)$. We know that for each $i\in [1,n]$, we have $v_{\q}(b_i)>v_{\q}(b_0)$. Let $\p=\q\cap \OO_K$. Then
$$v_{\p}(b_i)=\frac{1}{e(\q|\p)}v_{\q}(b_i)>\frac{1}{e(\q|\p)}v_{\q}(b_0)=v_{\p}(b_0).$$
This implies that $\p=\q\cap\OO_K\in X(K,\gamma)$. We conclude that
\[X(L,\gamma)=\{\q\in\Spec(\OO_L)\mid \q\cap\OO_K\in X(K,\gamma)\}.\qedhere\]
\end{proof}

Note that if $L\cap K(\gamma)\neq K$, then it is possible that
$$X(L,\gamma)\supsetneq\{\q\in\Spec(\OO_L)\mid \q\cap\OO_K\in X(K,\gamma)\}.$$
For example, consider $K=\Q$, $L=\Q[i]$ and $\gamma=\frac{1}{2+i}$. Then $X(L,\gamma)=\{(2+i)\}$, but $X(K,\gamma)=\emptyset$ so $\{\q\in\Spec(\OO_L)\mid \q\cap\OO_K\in X(K,\gamma)\}=\emptyset$.

\begin{lemma}\label{Y going up down}
If $K\subseteq L$ are number fields then
$$Y(K,\gamma)=\{\p\in\Spec(\OO_K)\mid \text{there exists }\q\in\Spec(\OO_L)\text{ above } \p \text{ such that } \q\in Y(L,\gamma)\},$$
$$Y(L,\gamma)\subseteq \{\q\in\Spec(\OO_L)\mid \q\cap\OO_K\in Y(K,\gamma)\}.$$
Moreover if $L\cap K(\gamma)=K$, then we have
$$Y(L,\gamma)= \{\q\in\Spec(\OO_L)\mid \q\cap\OO_K\in Y(K,\gamma)\}.$$
\end{lemma}
\begin{proof}
First, suppose $\p\in\Spec(\OO_K)$ for which there is a prime $\q\in\Spec(\OO_L)$ above $\p$ such that $\q\in Y(L,\gamma)$. We will show that $\p\in Y(K,\gamma)$.
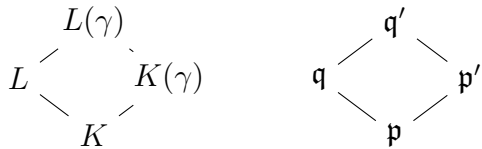
\begin{center}
\begin{tikzpicture}
\node (1) at (0,0) {$K$};
\node (2) at (-1,0.75) {$L$};
\node (3) at (1,0.75) {$K(\gamma)$};
\node (4) at (0,1.5) {$L(\gamma)$};
\node (5) at (4,0) {$\p$};
\node (6) at (3,0.75) {$\q$};
\node (7) at (5,0.75) {$\p'$};
\node (8) at (4,1.5) {$\q'$};
\draw [black] (1) -- (2);
\draw [black] (1) -- (3);
\draw [black] (4) -- (2);
\draw [black] (4) -- (3);
\draw [black] (5) -- (6);
\draw [black] (5) -- (7);
\draw [black] (8) -- (6);
\draw [black] (8) -- (7);
\end{tikzpicture}    
\end{center}
Since $\q\in Y(L,\gamma)$, there is a prime $\q'$ of $\OO_{L(\gamma)}$ above $\q$ for which $v_{\q'}(\gamma)<0$. Let $\p'=\q'\cap \OO_{K(\gamma)}$. Then $\p'$ is a prime of $\OO_{K(\gamma)}$ that is above $\p$. Moreover, $v_{\p'}(\gamma)=\frac{1}{e(\q'|\p')}v_{\q'}(\gamma)<0$. We see that $\p\in Y(K,\gamma)$. We have shown that
$$Y(K,\gamma)\supseteq\{\p\in\Spec(\OO_K)\mid \text{there exists }\q\in\Spec(\OO_L)\text{ above } \p \text{ such that } \q\in Y(L,\gamma)\}.$$

Next, consider a prime $\p\in Y(K,\gamma)$. This means there is a prime $\p'$ of $\OO_{K(\gamma)}$ above $\p$ for which $v_{\p'}(\gamma)<0$. Let $\q'$ be a prime of $\OO_{L(\gamma)}$ above $\p'$. Then $v_{\q'}(\gamma)=e(\q'|\p')v_{\p'}(\gamma)<0$. Let $\q=\q'\cap \OO_L$. Now $\q'$ is a prime of $\OO_{L(\gamma)}$ above $\q$ for which $v_{\q'}(\gamma)<0$. This means that $\q\in Y(L,\gamma)$. Moreover, $\q$ is a prime of $\OO_L$ above $\p$. We have shown that
$$Y(K,\gamma)\subseteq\{\p\in\Spec(\OO_K)\mid \text{there exists }\q\in\Spec(\OO_L)\text{ above } \p \text{ such that } \q\in Y(L,\gamma)\}.$$

Next, consider a prime $\q\in Y(L,\gamma)$. Let $\p=\q\cap \OO_K$. Then $\q$ is a prime of $\OO_L$ above $\p$ and $\q$ is in $Y(K,\gamma)$. Therefore from the previous part $\p\in Y(K,\gamma)$. This shows that
$$Y(L,\gamma)\subseteq\{\q\in\Spec(\OO_L)\mid \q\cap\OO_K\in Y(K,\gamma)\}.$$

For the last part, suppose $L\cap K(\gamma)=K$. This means that $$f_{L,\gamma}(x)=f_{K,\gamma}(x)=x^n+b_{n-1}x^{n-1}+\dots+b_0\in K(x).$$
Consider $\q\in \Spec(\OO_L)$ for which $\p=\q\cap\OO_K\in Y(K,\gamma)$. We know that there is some $i\in [0,n-1]$, for which $v_{\p}(b_i)<0$. Then $v_{\q}(b_i)=e(\q|\p)v_{\p}(b_i)<0$.
This implies that $\q\in Y(K,\gamma)$. We conclude that
\[Y(L,\gamma)=\{\q\in\Spec(\OO_L)\mid \q\cap\OO_K\in Y(K,\gamma)\}.\qedhere\]
\end{proof}
Note that if $L\cap K(\gamma)\neq K$, then it is possible that
$$Y(L,\gamma)\subsetneq\{\q\in\Spec(\OO_L)\mid \q\cap\OO_K\in Y(K,\gamma)\}.$$
For example, consider $K=\Q$, $L=\Q[i]$ and $\gamma=\frac{1}{2+i}$. Then $Y(L,\gamma)=\{(2+i)\}$, but $Y(K,\gamma)=\{(5)\}$ so $\{\q\in\Spec(\OO_L)\mid \q\cap\OO_K\in Y(K,\gamma)\}=\{(2+i),(2-i)\}$.

\begin{proposition}
If $K \subseteq L$ are number fields, then $$\OO_L[\gamma]\cap K=\OO_K[\gamma]\cap K.$$
\end{proposition}
\begin{proof}
Clearly $\okgm\cap K\subseteq \OO_L[\gamma]\cap K$. Suppose we have $\frac{a}{b} \in \OO_L[\gamma]\cap K$, $a, b \in \OO_K$. Since $\frac{a}{b} \in \OO_L[\gamma]$, by Proposition \ref{Determine OK[gamma]cap K} we know that whenever $\q \in \Spec(\OO_L)$ satisfies $v_\q(b)>v_\q(a)$, then $\q \in X(L,\gamma).$ Now suppose we have $\p \in \Spec(\OO_K)$ such that $v_\p(b)>v_\p(a)$. Then for any $\q \in \Spec(\OO_L)$ that is above $\p$, we also have $v_\q(b)>v_\q(a)$, which implies that $\q \in X(L,\gamma)$. Then by Lemma $\ref{Going up and down}$, we see that $\p \in X(K,\gamma)$. Therefore, $\frac{a}{b} \in \OO_K[\gamma]$ by Proposition $\ref{Determine OK[gamma]cap K}$.
\end{proof}

If two algebraic numbers $\gamma_1,\gamma_2$ satisfy $\OO_K[\gamma_1]\cap K=\OO_K$ and $\OO_K[\gamma_2]\cap K=\OO_K$, then we cannot always conclude that $\OO_K[\gamma_1,\gamma_2]\cap K=\OO_K$. For example, take $\gamma_1=\frac{1}{2+i}$, $\gamma_2=\frac{1}{2-i}$ and $K=\Q$. However, we will show that if $K(\gamma_1)\cap K(\gamma_2)=K$, then we can conclude that $\OO_K[\gamma_1,\gamma_2]\cap K=\OO_K$.

\begin{proposition}
Suppose we are given $\gamma_1,\dots,\gamma_n\in\overline{\Q}$ such that for every $1\leq m\leq n-1$ and for every $m+1\leq j\leq n$:
$$K(\gamma_1,\dots,\gamma_{m-1},\gamma_m)\cap K(\gamma_1,\dots,\gamma_{m-1},\gamma_j)=K(\gamma_1,\dots,\gamma_{m-1}).$$
Suppose moreover that for each $1\leq i\leq n$ we have $\OO_K[\gamma_i]\cap K=\OO_K$, then we have
$$\OO_K[\gamma_1,\dots,\gamma_n]\cap K=\OO_K.$$
\end{proposition}
\begin{proof}
We start by showing that $X(K(\gamma_1,\dots,\gamma_{k-1}),\gamma_{i})=\emptyset$ for all $1\leq k\leq n$ and $i \geq k$.
We will induct on $k$. The base case $k=1$ follows from Theorem \ref{X=empty result} and the assumption that each $\OO_K[\gamma_i]\cap K=\OO_K$. 
Now assume that this is true for $k-1$, that is, for each $i\geq k-1$, we have $X(K(\gamma_1,\dots,\gamma_{k-2}),\gamma_{i})=\emptyset$.
Since we have $$K(\gamma_1,\dots,\gamma_{k-2},\gamma_{k-1}) \cap K(\gamma_1,\dots,\gamma_{k-2},\gamma_{i})=K(\gamma_1,\dots,\gamma_{k-2}),$$ it follows from Lemma $\ref{Going up and down}$ that
\begin{equation*}
    \begin{split}
    &X(K(\gamma_1,\dots,\gamma_{k-1}),\gamma_{i})\\    
    =\{\q \in \Spec(\OO_{K(\gamma_1,\dots,\gamma_{k-1})})\mid \q \cap K(\gamma_1,\dots,\gamma_{k-2}) \in X(K(\gamma_1,\dots,\gamma_{k-2}),\gamma_{i})\}.
    \end{split}
\end{equation*}
Therefore, $X(K(\gamma_1,\dots,\gamma_{k-1}),\gamma_{i})=\emptyset$. This finishes the inductive step. In particular, we have shown that $X(K(\gamma_1,\dots,\gamma_{k-1}),\gamma_k)=\emptyset$ for each $1 \leq k \leq n$.

Now we will show that $\OO_K[\gamma_1,\dots,\gamma_n]\cap K\neq\OO_K$.
Assume for the sake of contradiction that $\OO_K[\gamma_1,\dots,\gamma_n]\cap K\neq\OO_K$. This means there is some $\beta\in \OO_K[\gamma_1,\dots,\gamma_n]\cap K$ for which $\beta\notin \OO_K$. Then there must be a prime $\p\in\Spec(\OO_K)$ for which $v_{\p}(\beta)<0$. Since $K(\beta)=K$, Theorem \ref{X cond equiv} implies that $\p\in X(K,\beta)$ and hence there is $\alpha\in \p$ for which $\frac{1}{\alpha}\in \OO_K[\beta]\subseteq \OO_K[\gamma_1,\dots,\gamma_n]\cap K$.

Next, we will show by induction on $k$ that there exist $\p_k \in \Spec(\OO_{K(\gamma_1,\dots,\gamma_k)})$, such that $\p_k\cap K=\p$ and $v_{\p_k}(\gamma_1),\dots,v_{\p_k}(\gamma_k)$ are all non-negative.
We start with the base case $k=1$. Since $X(K,\gamma_1)=\emptyset$, $\p \not \in X(K,\gamma_1)$. Therefore, there exists $\p_1 \in \Spec(\OO_{K(\gamma_1)})$ such that $\p_1\cap K=\p$ and $v_{\p_1}(\gamma_1)\geq 0$.
Now suppose we have constructed $\p_{k-1}$. Since $X(K(\gamma_1,\dots,\gamma_{k-1}),\gamma_k)=\emptyset$, $\p_{k-1} \not \in X(K(\gamma_1,\dots,\gamma_{k-1}),\gamma_{k})$. Therefore, there exists $\p_k \in\Spec( \OO_{K(\gamma_1,\dots,\gamma_{k})})$ such that $\p_k \cap K(\gamma_1,\dots,\gamma_{k-1})=\p_{k-1}$, and $v_{\p_k}(\gamma_{k}) \geq 0$. But since $\p_k$ is above $\p_{k-1}$ and $v_{\p_{k-1}}(\gamma_1),\dots,v_{\p_{k-1}}(\gamma_{k-1}) \geq 0$, we have that $v_{\p_k}(\gamma_1),\dots,v_{\p_k}(\gamma_{k-1}) \geq 0$. This finishes the inductive step.

In particular, if we set $k=n$, this shows that there exists $\p_n\in\Spec(\OO_{K(\gamma_1,\dots,\gamma_n)})$ such that $\p_n$ is above $\p$ and $v_{\p_n}(\gamma_i) \geq 0$ for $1 \leq i \leq n$. Since $\frac{1}{\alpha}\in \OO_K[\gamma_1,\dots,\gamma_n]$, we must have that $v_{\p_n}(\frac{1}{\alpha}) \geq 0$, so $v_{\p_n}(\alpha) \leq 0$. This contradicts the fact that $\alpha \in \p$ and $\p=\p_n \cap K$. Therefore, $\OO_K[\gamma_1,\dots,\gamma_n]\cap K=\OO_K$.
\end{proof}
\begin{corollary}
If $\OO_K[\gamma_1]\cap K=\OO_K[\gamma_2]\cap K=\OO_K$ and $K(\gamma_1)\cap K(\gamma_2)=K$ then
$$\OO_K[\gamma_1,\gamma_2]\cap K= \OO_K.$$
\end{corollary}

\section{The Generating Set of an Algebraic Number}\label{Sec Gen set}

In this section we will start by constructing a generating set of $\gamma$. After this we will study the properties of generating sets of $\gamma$. We will prove Proposition \ref{OK[gamma] in terms of S}, Proposition \ref{Zbar[gamma in terms of S]} and Theorem \ref{radical Dgamma 3 results}.

Given an algebraic number $\gamma\in\overline{\Q}$, and number field $K$, denote
$$\LL(\gamma,K)=\{\q\in X(K,\gamma)\mid \text{for every } K_1\subsetneq K\text{ we have } \q\cap K_1\notin X(K_1,\gamma)\}.$$
Further, we define $$\mathcal{L}(\gamma)=\{K\mid \LL(\gamma,K)\neq\emptyset\}.$$

\begin{lemma}\label{containment LL(gamma,K)}
If $K\in\LL(\gamma)$, then we have $K\subseteq\Q(\gamma)$.
\end{lemma}
\begin{proof}
Pick a prime $\q\in \LL(\gamma,K)$.
Let $K_1=K\cap \Q(\gamma)$. Now we have $K_1(\gamma)\cap K=K_1$ and $\q\in X(K,\gamma)$, so by Lemma \ref{Going up and down} we get $\q\cap K_1\in X(K_1,\gamma)$. Since $\q\in\LL(\gamma,K)$ and $K_1\subseteq K$, we conclude that $K_1=K$. This shows that $K\subseteq \Q(\gamma)$.
\end{proof}

\begin{corollary}
For any algebraic number $\gamma$, $\#\LL(\gamma)<\infty$.
\end{corollary}

We will use $\LL(\gamma)$ to construct a generating set of $\gamma$.
Denote $J_{K,\gamma}=\prod_{\q\in\LL(\gamma,K)}\q$, it is an ideal of $\OO_K$.
For each $K\in\LL(\gamma)$, denote the class number of $K$ as $h_{K}$. Then $J_{K,\gamma}^{h_K}$ is a principal ideal, let $\alpha_{K}\in\OO_K$ be one of its generators ($\alpha_{K}$ is not uniquely determined, but pick one).

\begin{lemma}\label{Q(alpha K)}
Let $K\in\LL(\gamma)$ and $\alpha_K$ be as above. Then $\Q(\alpha_K)=K$.
\end{lemma}
\begin{proof}
Say $\Q(\alpha_K)=K_1$, so $K_1\subseteq K$. Pick a prime $\q_1\in \LL(\gamma,K)$. Let $\p=\q_1\cap K_1\in\Spec(\OO_{K_1})$. Note that if we show that $\p\in X(K,\gamma)$, then the definition of $\LL(\gamma,K)$ will imply that $K_1=K$. Now
$$\alpha_K\OO_K=J_{K,\gamma}^{h_K}=\prod_{\q\in \LL(\gamma,K)}\q^{h_K},$$
so $\alpha_K\in \q_1$ and hence $\alpha_K\in\p$. Next let $\q_2$ be a prime of $\OO_K$ above $\p$, then $\alpha_K\in\p\OO_K\subseteq \q_2$. This implies that $\q_2$ divides $\alpha_K\OO_K$ and hence $\q_2\in \LL(\gamma,K)\subseteq X(K,\gamma)$. We have shown that every prime of $\OO_K$ above $\p$ is in $X(K,\gamma)$, hence Lemma \ref{Going up and down} implies that $\p\in X(K_1,\gamma)$. But $\p=\q_1\cap K_1$ and $\q_1\in \LL(\gamma,K)$, so we can conclude that $K_1=K$, meaning $\Q(\alpha_K)=K$.
\end{proof}

Let $S=\{\alpha_{K}\mid K\in\LL(\gamma)\}$. Note that $|S|=|\LL(\gamma)|<\infty$.

\begin{corollary}\label{prod K}
Let $S$ be as above, then $\Q(S)$ is the composite of the fields in $\LL(\gamma)$, that is,
$$\prod_{K\in\LL(\gamma)}K=\Q(S)\subseteq \Q(\gamma).$$
\end{corollary}

\begin{proposition}\label{X(L gamma) q int S}
Let $S=\{\alpha_{K}\mid K\in\LL(\gamma)\}$ as above. Then $S$ is a generating set of $\gamma$ with $|S|=\#\LL(\gamma)$.
\end{proposition}
\begin{proof}
We know that $|S|=\#\LL(\gamma)$, so we only need to show that for every number field $L$ we have
$$X(L,\gamma)=\{\q\in\Spec(\OO_L)\mid \q\cap S\neq\emptyset\}.$$
Fix a number field $L$. Consider a prime $\q\in\Spec(\OO_L)$ for which $\q\cap S\neq \emptyset$. Say $\alpha_K\in \q$, $K\in\LL(\gamma)$. Now $\alpha_K\in \q\subseteq L$, so $K=\Q(\alpha_K)\subseteq L$.
Let $\q''=\q\cap K$, so $\alpha_K\in\q''$. Since
$$\alpha_K\OO_K=J_{K,\gamma}^{h_K}=\prod_{\p\in \LL(\gamma,K)}\p^{h_K},$$
we see that $\q''\in\LL(\gamma,K)\subseteq X(K,\gamma)$. Then by Lemma \ref{Going up and down} we see that $\q\in X(L,\gamma)$.

For the other direction, consider a prime $\q\in X(L,\gamma)$. Let $K_1=L\cap \Q(\gamma_1)$ and $\q'=\q\cap K_1$. Then $K_1(\gamma)\cap L=K_1$, so by Lemma \ref{Going up and down} we see that $\q'\in X(K_1,\gamma)$.
\begin{center}
\begin{tikzpicture}
\node (1) at (0,0) {$\Q$};
\node (2) at (0,0.9) {$\Q(\alpha_K)=K$};
\node (3) at (0,1.8) {$L \cap \Q(\gamma)=K_1$};
\node (4) at (-1,2.6) {$L$};
\node (5) at (1,2.6) {$\Q(\gamma)$};
\node (6) at (4,0) {$\q \cap \mathbb{Z}$};
\node (7) at (4,0.9) {$\q''=\q \cap K$};
\node (8) at (4,1.8) {$\q'=\q \cap K_1$};
\node (9) at (3,2.6) {$\q$};
\node (10) at (5,2.6) {$\p$};
\draw [black] (1) -- (2);
\draw [black] (2) -- (3);
\draw [black] (3) -- (4);
\draw [black] (3) -- (5);
\draw [black] (6) -- (7);
\draw [black] (7) -- (8);
\draw [black] (8) -- (9);
\draw [black] (8) -- (10);
\end{tikzpicture}
\end{center}
Now, consider the set
$$\{E\mid E\subseteq K_1, \q'\cap E\in X(E,\gamma)\}.$$
This is a finite collection of fields, so let $K$ be a minimal field in this collection. Let $\q''=\q'\cap K=\q\cap K$. It follows from minimality of $K$ that $\q''\in \LL(\gamma,K)$ and hence $K\in\LL(\gamma)$. From the definition of $\alpha_K$, we see that $\alpha_K\in\q''$. We conclude that $\alpha_K\in \q$ and hence $\q\cap S\neq\emptyset$.
\end{proof}

\begin{example}
Consider $\gamma=\frac{1}{60+15i}$. Then $X(\Q,\gamma)=\{3,5\}$ and $X(\Q(i),\gamma)=\{3,2+i,2-i,4+i\}$. Therefore $\LL(\gamma,\Q)=\{3,5\}$, $\LL(\gamma,\Q(i))=\{4+i\}$ and $\LL(\gamma)=\{\Q,\Q(i)\}$. We have $\alpha_{\Q}=15$ and $\alpha_{\Q(i)}=4+i$. Therefore for every number field $L$,
$$X(L,\gamma)=\{\q\in\Spec(\OO_L)\mid \{15,4+i\}\cap S\neq \emptyset\}.$$
\end{example}

\begin{proposition}\label{Q adjoin generating set}
Suppose $S'$ is a generating set of $\gamma$.
Then we have $\#\LL(\gamma)\leq |S'|$ and the composite of the fields in $\LL(\gamma)$ is contained in $\Q(S')$.
Moreover, if $|S'|=\#\LL(\gamma)$, then we have $$\prod_{K\in\LL(\gamma)}K=\Q(S')\subseteq \Q(\gamma).$$
\end{proposition}
\begin{proof}
Consider $K\in \LL(\gamma)$ and $\p\in \LL(\gamma,K)$. Then $\p\in X(K,\gamma)$, so $\p\cap S'\neq\emptyset$. Say $\beta_K\in \p\cap S'$. Now let $K_1=\Q(\beta_K)$ and $\p'=\p\cap K_1$. Then $\beta_K\in \p'\cap S'$, so $\p'\in X(K_1,\gamma)$. But then the definition of $\LL(\gamma,K)$ implies that $K_1=K$, meaning that $\Q(\beta_K)=K$. This shows that for distinct $K\in\LL(\gamma)$, we have distinct $\beta_K\in S'$. So we conclude that $|\LL(\gamma)|\leq |S'|$. Moreover, the fact that $\beta_K\in S'$ and $\Q(\beta_K)=K$ implies that the composite
$\prod_{K\in \LL(\gamma)}K\subseteq \Q(S')$. Finally if $|S'|=|\LL(\gamma)|$, then $S'=\{\beta_K\mid K\in \LL(\gamma)\}$ and hence $\Q(S')=\prod_{K\in\LL(\gamma)}K$.
\end{proof}

\begin{corollary}
Suppose $S_1$ and $S_2$ are generating sets of $\gamma_1$ and $\gamma_2$ respectively with $|S_1|=\#\LL(\gamma_1)$ and $|S_2|=\#\LL(\gamma_2)$. Suppose further that $\gamma_1,\gamma_2$ have the property that for every number field $L$, $X(L,\gamma_1)=X(L,\gamma_2)$. Then $\Q(S_1)=\Q(S_2)$.
\end{corollary}
\begin{proof}
The property that for every number field $L$, $X(L,\gamma_1)=X(L,\gamma_2)$, implies that $\LL(\gamma_1)=\LL(\gamma_2)$. Now we see that
\[\Q(S_1)=\prod_{K\in\LL(\gamma_1)}K=\Q(S_2).\qedhere\]
\end{proof}

\begin{customprop}{\ref{OK[gamma] in terms of S}}
For a subset $S\subseteq\Zbar$,
$S$ is a generating set of $\gamma$ if and only if for every number field $K$, we have
$$\OO_K[\gamma]\cap K=\OO_K\left[1/\alpha\mid \alpha\in S\cap K\right].$$
\end{customprop}
\begin{proof}
First suppose that $S$ is a generating set of $\gamma$.
Fix a number field $K$. Consider $\beta \in \OO_K[\gamma]\cap K$. Then by proposition $\ref{Determine OK[gamma]cap K}$, $$\{\p \in \Spec(\OO_K)\mid v_\p(\beta)<0\} \subseteq X(K,\gamma).$$
But by the definition of generating set, $X(K,
\gamma)=\{\p \in \Spec(\OO_K)\mid\p \cap S \neq \emptyset\}$. Let $\p_1,\dots,\p_n$ be the primes in $\Spec(\OO_K)$ such that $v_{\p_i}(\beta)<0$. Then for each $\p_i$, $\p_i \cap S \neq \emptyset$. Consider $\alpha_i \in \p_i \cap S$, so $\alpha_i\in S\cap K$. Then for larger enough $N$, we have $\alpha_1^N\dots\alpha_n^N \beta\in \OO_K$, so $\beta \in \OO_K\left[\frac{1}{\alpha}\mid \alpha\in S\cap K\right]$.

Now consider $\alpha \in S\cap K$. Then for every $\p \in \Spec(\OO_K)$ such that $v_{\p}(\alpha)>v_{\p}(1)$, we have $\alpha \in S \cap \p$, so $S\cap \p \neq \emptyset$ and $\p\in X(K,\gamma)$. Then by proposition $\ref{Determine OK[gamma]cap K}$, we see that $\frac{1}{\alpha} \in \OO_K[\gamma]$. Therefore, we have $\OO_K[\gamma]\cap K=\OO_K\left[\frac{1}{\alpha}\mid \alpha\in S\cap K\right].$

Conversely, suppose for every number field $K$, we have
$$\OO_K[\gamma]\cap K=\OO_K\left[1/\alpha\mid \alpha\in S\cap K\right].$$
 We want to show that $S$ is a generating set, that is, we want to show that for every number field $K$, $X(K,\gamma)=\{\p \in \Spec\OO_K\mid \p \cap S \neq \emptyset\}$.

Suppose we have $\p \in \Spec(\OO_K)$ with $\p\cap S\neq\emptyset$. Say $\alpha \in \p \cap S$. Then $\alpha \in S \cap K$. Since $\OO_K[\gamma]\cap K=\OO_K\left[\frac{1}{\alpha}\mid \alpha\in S\cap K\right],$ we have $\frac{1}{\alpha} \in \OO_K[\gamma] \cap K$, and $v_\p(1) < v_\p(\alpha)$. Then by proposition \ref{Determine OK[gamma]cap K}, $\p \in X(K,\gamma)$.

On the other hand, suppose $\p \in X(K,\gamma)$. By Theorem \ref{X cond equiv}, there exists $\beta \in \p$ such that $\frac{1}{\beta} \in \OO_K[\gamma]$. Then $\frac{1}{\beta} \in \OO_K[\gamma] \cap K=\OO_K[\frac{1}{\alpha} \mid \alpha \in S \cap K]$. Therefore, $\frac{1}{\beta}=f(\frac{1}{\alpha_1},\dots,\frac{1}{\alpha_m})$, where $f(x_1,\dots,x_m) \in \OO_K[x_1,\dots,x_m]$ and $\alpha_i \in S \cap K$. We can clear denominators by multiply both sides by high enough power of $\alpha_1\dots\alpha_m$. We get that
$$(\alpha_1\dots\alpha_m)^N=\beta g(\alpha_1,\dots,\alpha_m),$$
and $g(\alpha_1,\dots,\alpha_m)$ is in $\OO_K$. We see that $(\alpha_1\dots\alpha_m)^N \in \p$. So there exists $i$ such that $\alpha_i \in \p$. Therefore $\alpha_i \in \p \cap S$, and hence $\p \cap S \neq \emptyset$.
\end{proof}

\begin{customprop}{\ref{Zbar[gamma in terms of S]}}
Suppose $S$ is a generating set of $\gamma$. Then we have
$$\Zbar[\gamma]=\Zbar\left[1/\alpha\mid \alpha\in S\right].$$
\end{customprop}
\begin{proof}
Pick $\alpha\in S$. Let $K=\Q(\alpha)$. Then by Proposition \ref{OK[gamma] in terms of S}, we know that $\frac{1}{\alpha}\in \OO_K[\gamma]\cap K$ and hence $\frac{1}{\alpha}\in\Zbar[\gamma]$. This shows that $\Zbar\left[\frac{1}{\alpha}\mid \alpha\in S\right]\subseteq\Zbar[\gamma]$.

For the other direction suppose $\beta\in \Zbar[\gamma]$. Say $\beta=\sum_{i=0}^m c_i\gamma^i$ for $c_i\in\Zbar$. Let $K=\Q(c_0,\dots,c_m,\beta)$. Therefore $\beta\in\OO_K[\gamma]\cap K$. Now Proposition \ref{OK[gamma] in terms of S} implies that $\beta\in \OO_K\left[\frac{1}{\alpha}\mid \alpha\in S\cap K\right]$ and hence $\beta\in \Zbar\left[\frac{1}{\alpha}\mid \alpha\in S\right]$. This shows that $\Zbar[\gamma]\subseteq\Zbar\left[\frac{1}{\alpha}\mid \alpha\in S\right]$.
\end{proof}

\begin{lemma}\label{Radical of Dgamma}
Let $S$ be a finite generating set of $\gamma$, then as ideals of $\overline{\Z}$, we have
$$\sqrt{\D_{\gamma}}=\sqrt{\Big(\prod_{\alpha\in S}\alpha\Big)\overline{\Z}}.$$
\end{lemma}
\begin{proof}
First consider $\beta\in \D_{\gamma}$, say $\beta^{m}\in\D_{\gamma}$. Therefore $\beta^{m}\gamma\in \overline{\Z}$. Consider $L=\Q(\beta,\gamma,S)$. We want to show that there is some $N$ for which $\beta^{N}\in (\prod_{\alpha\in S}\alpha)\Zbar$. This is equivalent to showing that $\frac{\beta^{N}}{\prod_{\alpha\in S}\alpha}\in\OO_L$.
Now for a prime $\p\in\Spec(\OO_L)$, if $v_{\p}(\prod_{\alpha\in S}\alpha)>0$, then there is some $\alpha\in S$ such that $\alpha\in\p$. By definition of generating set of $\gamma$, this implies that $\p\in X(L,\gamma)$. Now since $L(\gamma)=L$, we see that $v_{\p}(\gamma)<0$. Moreover $\beta^m\gamma\in\Zbar$ implies $v_{\p}(\beta)>0$. We have shown that $v_{\p}(\prod_{\alpha\in S}\alpha)>0$ implies $v_{\p}(\beta)>0$. Therefore for a sufficiently large $N$, we will have $\frac{\beta^{N}}{\prod_{\alpha\in S}\alpha}\in\OO_L$. This shows that $\beta\in \sqrt{\left(\prod_{\alpha\in S}\alpha\right)\overline{\Z}}$.

For the other direction consider $\beta\in\sqrt{\left(\prod_{\alpha\in S}\alpha\right)\overline{\Z}}$. This means that for some $m$, $\frac{\beta^m}{\prod_{\alpha\in S}\alpha}\in\Zbar$. Again consider $L=\Q(\beta,\gamma,S)$. Therefore, we have $\frac{\beta^m}{\prod_{\alpha\in S}\alpha}\in\OO_L$. Consider $\p\in\Spec(\OO_L)$ for which $v_{\p}(\gamma)<0$. Then $\p\in X(L,\gamma)$ and hence by definition of generating set, $S\cap\p\neq\emptyset$. This implies $v_{\p}(\prod_{\alpha\in S}\alpha)>0$. Since $\frac{\beta^m}{\prod_{\alpha\in S}\alpha}\in\OO_L$, we see that $v_{\p}(\beta)>0$. We have shown that whenever a prime $\p\in\Spec(\OO_L)$ satisfies $v_{\p}(\gamma)<0$, then it also satisfies $v_{\p}(\beta)>0$. Therefore for sufficiently large $N$, we will have $\beta^N\gamma\in\OO_L$. This means that $\beta\in\sqrt{\D_{\gamma}}$.
\end{proof}

\begin{lemma}\label{prod Y}
Given a number field $K$, as ideals of $\OO_K$ we have
$$\sqrt{\D_{\gamma}\cap \OO_K}=\prod_{\p\in Y(K,\gamma)}\p.$$
\end{lemma}
\begin{proof}
First suppose $\beta\in \sqrt{\D_{\gamma}\cap\OO_K}$, say $\beta^N\in \D_{\gamma}\cap \OO_K$ so $\beta^N\gamma\in\Zbar$. Consider $\p\in Y(K,\gamma)$, so there is a prime $\q\in\Spec(\OO_{K(\gamma)})$ such that $\q\cap K=\p$ and $v_{\q}(\gamma)<0$. This implies $v_{\q}(\beta)>0$. Since $\beta\in K$, we see that $v_{\p}(\beta)>0$. We have shown that $v_{\p}(\beta)>0$ for all $\p\in Y(K,\gamma)$. This shows that $\beta\in\prod_{\p\in Y(K,\gamma)}\p$.

For the other direction, suppose $\beta\in\prod_{\p\in Y(K,\gamma)}\p$. Let $\q\in\Spec(\OO_{K(\gamma)})$ be a prime for which $v_{\q}(\gamma)<0$. Then $\p=\q\cap K\in Y(K,\gamma)$. Therefore, $v_{\p}(\beta)>0$ and hence $v_{\q}(\beta)>0$. We have shown that for every prime $\q\in\Spec(\OO_{K(\gamma)})$ if $v_{\q}(\gamma)<0$ then $v_{\q}(\beta)>0$. Therefore for sufficiently large $N$, we have $\beta^N\gamma\in \OO_{K(\gamma)}$. This shows that $\beta^N\in\D_{\gamma}$ and hence $\beta\in \sqrt{\D_{\gamma}\cap\OO_K}$.
\end{proof}

\begin{lemma}\label{prod p in X}
Let $S$ be a finite generating set of $\gamma$. Then for any number field $K$, as ideals of $\OO_K$ we have
$$\sqrt{\Big(\prod_{\alpha\in S\cap K}\alpha\Big)\OO_K}=\prod_{\p\in X(K,\gamma)}\p.$$
\end{lemma}
\begin{proof}
We know that
$$\sqrt{\Big(\prod_{\alpha\in S\cap K}\alpha\Big)\OO_K}=\bigcap_{\left(\prod_{\alpha\in S\cap K}\alpha\right)\in\p}\p.$$
Now the primes that contain $\left(\prod_{\alpha\in S\cap K}\alpha\right)$ are precisely those that contain at least one $\alpha\in S\cap K.$ By the definition of generating set of $\gamma$, these are the primes in $X(K,\gamma)$. This is a finite set of primes and hence
\[\bigcap_{\p\in X(K,\gamma)}\p=\prod_{\p\in X(K,\gamma)}\p.\qedhere\]
\end{proof}

\begin{proof}[Proof of Theorem \ref{radical Dgamma 3 results}]
Follows from Lemma \ref{Radical of Dgamma}, Lemma \ref{prod Y} and Lemma \ref{prod p in X}.
\end{proof}

\begin{theorem}\label{X(gamma1=Xgamma2)}
For $\gamma_1,\gamma_2\in \overline{\Q}$, the following are equivalent:
\begin{enumerate}
    \item $X(\Q(\gamma_1,\gamma_2),\gamma_1)=X(\Q(\gamma_1,\gamma_2),\gamma_2)$.
    \item For every number field $L$, we have $X(L,\gamma_1)=X(L,\gamma_2)$.
    \item $\sqrt{\D_{\gamma_1}}=\sqrt{\D_{\gamma_2}}$.
    \item For every number field $L$, we have $Y(L,\gamma_1)=Y(L,\gamma_2)$.
    \item $Y(\Q(\gamma_1,\gamma_2),\gamma_1)=Y(\Q(\gamma_1,\gamma_2),\gamma_2)$.
\end{enumerate}
\end{theorem}
\begin{proof}
Suppose $\gamma_1,\gamma_2$ satisfy the condition $(1)$. Consider a number field $L$ and let $K=L \cap \Q(\gamma_1, \gamma_2)$. Then by Lemma \ref{Going up and down} we have
\begin{equation*}
\begin{split}
    X(K,\gamma_1)&=\{\p\in\Spec(\OO_K)\mid \forall\q\in\Spec(\OO_{\Q(\gamma_1,\gamma_2)}), \text{ if } \q\cap\OO_K=\p \text{ then } \q\in X(\Q(\gamma_1,\gamma_2),\gamma_1)\}\\
    &=\{\p\in\Spec(\OO_K)\mid \forall\q\in\Spec(\OO_{\Q(\gamma_1,\gamma_2)}), \text{ if } \q\cap\OO_K=\p \text{ then } \q\in X(\Q(\gamma_1,\gamma_2),\gamma_2)\}\\
    &=X(K,\gamma_2).
\end{split}
\end{equation*}
Moreover we have $K(\gamma_1)\cap L=K$, $K(\gamma_2)\cap L=K$ and hence by Lemma \ref{Going up and down} we have
\begin{equation*}
    \begin{split}
        X(L,\gamma_1)&=\{\q\in\Spec(\OO_L)\mid \q\cap\OO_K\in X(K,\gamma_1)\}\\
        &=\{\q\in\Spec(\OO_L)\mid \q\cap\OO_K\in X(K,\gamma_2)\}\\
        &=X(L,\gamma_2).
    \end{split}
\end{equation*}
Therefore $\gamma_1,\gamma_2$ satisfy condition $(2)$.

Next suppose $\gamma_1,\gamma_2$ satisfy condition $(2)$. Let $S$ be a finite generating set of $\gamma_1$. Then $S$ is a finite generating set of $\gamma_2$ as well. We conclude from Lemma \ref{Radical of Dgamma} that $\sqrt{\D_{\gamma_1}}=\sqrt{\D_{\gamma_2}}$. Therefore $\gamma_1,\gamma_2$ satisfy condition $(3)$.

Next suppose $\gamma_1,\gamma_2$ satisfy condition $(3)$. Then for any number field $L$, we have
$$\sqrt{\D_{\gamma_1}\cap\OO_L}=\sqrt{\D_{\gamma_1}}\cap \OO_L=\sqrt{\D_{\gamma_2}}\cap\OO_L=\sqrt{\D_{\gamma_2}\cap\OO_L}.$$
Now Lemma \ref{prod Y} implies $Y(L,\gamma_1)=Y(L,\gamma_2)$.
Therefore $\gamma_1,\gamma_2$ satisfy condition $(4)$.

Condition $(4)$ clearly implies condition $(5)$.

Finally suppose $\gamma_1,\gamma_2$ satisfy condition $(5)$. Let $K=\Q(\gamma_1,\gamma_2)$. Then $K(\gamma_1)=K$, so we have
$$X(K,\gamma_1)=\{\p\in\Spec(\OO_K)\mid v_{\p}(\gamma_1)<0\}=Y(K,\gamma_1).$$
Similarly we see that $X(K,\gamma_2)=Y(K,\gamma_2)$. Therefore $X(\Q(\gamma_1,\gamma_2),\gamma_1)=X(\Q(\gamma_1,\gamma_2),\gamma_2)$ and $\gamma_1,\gamma_2$ satisfy condition $(1)$.
\end{proof}

\section{Local case}
\begin{proposition}\label{local case}
Let $K$ be a finite extension of $\Q_p$ and $\gamma\in\overline{\Q_p}$. Let $v$ be the unique valuation on $\overline{\Q_p}$ that extends $v_p$. Then
\begin{itemize}
    \item If $v(\gamma)\geq 0$, then $\OO_K[\gamma]\cap K=\OO_K$.
    \item If $v(\gamma)<0$, then $\OO_K[\gamma]\cap K=K$.
\end{itemize}
\end{proposition}
\begin{proof}
If $v(\gamma) \geq 0$, let $\Tilde{v}$ be the restriction of $v$ to $K(\gamma)$ and $v_K$ be the restriction of $v$ to $K$. Then for $\frac{a}{b} \in K \cap \OO_K[\gamma]$, we have that $\Tilde{v}(\frac{a}{b}) \geq 0$. Then $v_K(\frac{a}{b}) \geq 0$, so $\frac{a}{b} \in \OO_K$. So for $v(\gamma) \geq 0$, we have $\OO_K[\gamma] \cap K=\OO_K$.

Now suppose $v(\gamma)<0$. First we will show that $\OO_K \subsetneq \OO_K[\gamma] \cap K$. Consider $a=\Nm^{K(\gamma)}_K(\gamma)$.
Notice that
$$v(\gamma)=\Tilde{v}(\gamma)=\frac{1}{n}v_K\left(\Nm^{K(\gamma)}_K(\gamma)\right)=\frac{1}{n}v(a).$$
Thus $nv(\gamma)=v(a)$. If $\gamma_i$ is a conjugate of $\gamma$ over $K$, then we similarly see that $nv(\gamma_i)=v(a)$.
Now let $x^n+a_{n-1}x^{n-1}+ \dots +a_0=0$ be the minimal polynomial of $\gamma$ over $K$. So $a_0=\pm a$. Then we have that $v(a_i) \geq (n-i)v(\gamma)$ and $v(a)=nv(\gamma)$, so $v(a_i)>v(a)$. Say $a_{i_0}=\min_{1\leq i\leq n}(v(a_i))$ and $u=a_{i_0}^{-1}$. Now,
$$ua=\pm(ua_n\gamma^n+ ua_{n-1}\gamma^{n-1} \dots +ua_1\gamma).$$
For $1\leq i\leq n$, we have $v(a_{i_0})\leq v(a_i)$, so $v(ua_i)\geq 0$ and hence $ua_i\in \OO_K$. This shows that $ua\in \OO_K[\gamma]\cap K$. Moreover $v(a)<v(a_{i_0})$, so $v(ua)<0$ and $ua\notin \OO_K$. Let $b=ua$.

Finally we will show that $\OO_K[\gamma] \cap K=K$. Say $v_K(b)=\Tilde{v}(b)=-m$.
Let $t$ be a uniformizer of $\OO_K$, then $\alpha=t^{m-1}b\in \OO_K[\gamma]\cap K$ and $v_K(\alpha)=-1$.
Then for any $\beta \in K$, suppose $v_K(\beta) = -s$. Then $v_K(\frac{\beta}{\alpha^{s}})=0$, so $\frac{\beta}{\alpha^{s}} \in \OO_K^*$. Then $\beta =\frac{\beta}{\alpha^s} \alpha^{s} \in \OO_K[\gamma] \cap K$. Since $\beta$ is an arbitrary element in $K$, we have that $K= \OO_K[\gamma] \cap K$.
\end{proof}

\section{Relationships among $c(\gamma),d(\gamma)$, $e(\gamma)$ and $\deg(\gamma)$}\label{Sec: classify tuples}

\begin{proposition}\label{c,d,e,n nec cond}
Suppose $\gamma$ is an algebraic number of degree $n$. Then $c(\gamma)$ divides $d(\gamma)^{n-1}e(\gamma)$.
Moreover, all rational primes $p$ satisfy at least one of the following:
\begin{enumerate}
    \item $v_{p}(d(\gamma))+v_p(e(\gamma))\leq v_p(c(\gamma))$.
    \item $\frac{n-1}{n} v_p(c(\gamma))<v_p(e(\gamma))\leq v_p(c(\gamma))$ and $v_p(c(\gamma))=nv_p(d(\gamma))$.
\end{enumerate}
\end{proposition}
\begin{proof}
We know that there is some $1\leq j\leq n-1$ such that $v_p(a_j)=v_p(e(\gamma))$. We know that $\left\lceil\frac{v_p(a_n)-v_p(a_j)}{n-j}\right\rceil\leq v_p(d(\gamma))$, and hence
$$\frac{v_p(c(\gamma))-v_p(e(\gamma))}{n-j} =\frac{v_p(a_n)-v_p(a_j)}{n-j}\leq v_p(d(\gamma)).$$
This means that $v_p(c(\gamma))\leq (n-1)v_p(d(\gamma))+v_p(e(\gamma))$. Therefore $c(\gamma)|d(\gamma)^{n-1}e(\gamma)$.

First consider a rational prime $p$ that does not divide $e(\gamma)$. Since $d(\gamma)|c(\gamma)$, we have $v_{p}(d(\gamma))\leq v_{p}(c(\gamma))$. Therefore $v_{p}(d(\gamma))+v_{p}(e(\gamma))\leq v_{p}(c(\gamma))$.

Next consider a rational prime $p$ that divides $e(\gamma)$. In this case, $p$ also divides $c(\gamma)=a_n$ and $d(\gamma)$ (since $c(\gamma)|d(\gamma)^n$). Moreover $p$ divides $a_i$ for $1\leq i\leq n$ and hence $p\not|a_0$. We know that that
$$v_p(d(\gamma))=\max\left(0,\max_{0\leq j\leq n-1}\left\lceil\frac{v_p(a_n)-v_p(a_j)}{n-j}\right\rceil\right).$$
Now $v_p(a_0)=0$, so $\left\lceil\frac{v_p(a_n)-v_p(a_0)}{n-0}\right\rceil=\left\lceil\frac{v_p(c(\gamma))}{n}\right\rceil$.
Moreover for $1\leq j\leq n-1$,
    $$\frac{v_p(a_n)-v_p(a_j)}{n-j}\leq \frac{v_p(a_n)-v_p(e(\gamma))}{n-j}\leq v_p(a_n)-v_p(e(\gamma)).$$
    Since $v_p(a_n)-v_p(e(\gamma))$ is an integer, we conclude that $\left\lceil\frac{v_p(a_n)-v_p(a_j)}{n-j}\right\rceil\leq v_p(c(\gamma))-v_p(e(\gamma)).$
We also know that $v_p(d(\gamma))\neq 0$. Therefore
$$v_p(d(\gamma))=\max\left(0,\max_{0\leq j\leq n-1}\left\lceil\frac{v_p(a_n)-v_p(a_j)}{n-j}\right\rceil\right)
\leq \max\left(\left\lceil\frac{v_p(c(\gamma))}{n}\right\rceil,v_p(c(\gamma))-v_p(e(\gamma))\right).$$
\begin{itemize}
    \item Case 1: $v_p(c(\gamma))-v_p(e(\gamma))\geq\left\lceil\frac{v_p(c(\gamma))}{n}\right\rceil$. We see that $v_p(d(\gamma))\leq v_p(c(\gamma))-v_{p}(e(\gamma))$.
    \item Case 2: $v_p(c(\gamma))-v_p(e(\gamma))<\left\lceil\frac{v_p(c(\gamma))}{n}\right\rceil$.
    It follows that $v_p(d(\gamma))=\left\lceil\frac{v_p(c(\gamma))}{n}\right\rceil$ and hence $nv_p(d(\gamma))\geq v_p(c(\gamma))$. Moreover since $c(\gamma)|d(\gamma)^n$, we always have $v_p(c(\gamma))\leq nv_p(d(\gamma))$. Therefore $v_p(c(\gamma))= nv_p(d(\gamma))$.
    
    Since $e(\gamma)|c(\gamma)$, we have $v_p(e(\gamma))\leq v_p(c(\gamma))$. Finally $v_p(c(\gamma))-v_p(e(\gamma))<\left\lceil\frac{v_p(c(\gamma))}{n}\right\rceil$ implies $v_p(c(\gamma))-v_p(e(\gamma))<\frac{v_p(c(\gamma))}{n}$, which means $\frac{n-1}{n} v_p(c(\gamma))<v_p(e(\gamma))$.\qedhere
\end{itemize}
\end{proof}

\begin{proposition}\label{c,d,e,n suf cond}
Suppose we are give positive integers $c,d,e,n$ such that $d|c$, $c|d^n$, $e|c$, $c|d^{n-1}e$ and $n\geq 2$.
Suppose further that all rational primes $p$ satisfy at least one of the following:
\begin{enumerate}
    \item $v_{p}(d)+v_p(e)\leq v_p(c)$.
    \item $v_p(c)= nv_p(d)$.
\end{enumerate}
Then there is an algebraic number $\gamma$ for which $\deg(\gamma)=n$, $c(\gamma)=c$, $d(\gamma)=d$ and $e(\gamma)=e$.
\end{proposition}
\begin{proof}
Let $Y$ be the collection of primes factors of $c$. Pick a prime $q$ that is not in $Y$. We will construct integers $a_0,a_1,\dots,a_n$. Firstly set $a_n=c$. For primes $p$ that are not in $Y\cup\{q\}$, set $v_p(a_i)=0$ for each $0\leq i\leq n-1$. Set $v_q(a_i)=1$ for $0\leq i\leq n-1$. Finally for primes $p\in Y$, set $v_p(a_0)=0$ and for $1\leq i\leq n-1$ set $v_p(a_{i})=\max (v_p(c)-(n-i)v_p(d),v_p(e))$

Consider the polynomial $F(x)=a_nx^n+\dots+a_0$. By applying the Eisenstein criterion with prime $q$, we see that $F$ is irreducible. Let $\gamma$ be a root of $F$, then $\deg(\gamma)=n$, $c(\gamma)=a_n=c$.

For primes $p\in Y$, we know that $v_p(c)\leq (n-1)v_p(d)+v_p(e)$. This implies that $v_p(a_{1})=v_p(e)$. We also know that for $1\leq i\leq n-1$, $v_p(a_i)\geq v_p(e)$. Finally since $e|c$, we have $v_p(a_n)\geq v_p(e)$. This shows that $v_p(e(\gamma))=v_p(e)$. Moreover, for primes $p\notin Y$, $p$ does not divide $a_n$. Therefore $e(\gamma)=e$.

For a prime $p\in Y$ and $1\leq j\leq n-1$, we know that $v_p(a_{i})\geq v_p(c)-(n-i)v_p(d)$. This means that $\frac{v_p(c)-v_p(a_i)}{n-i}\leq v_p(d)$ and hence $\left\lceil\frac{v_p(c)-v_p(a_i)}{n-i}\right\rceil\leq v_p(d)$. Since $c|d^n$ and $v_p(a_0)=0$, we know that $\frac{v_p(c)-v_p(a_0)}{n}\leq v_p(d)$ and hence $\left\lceil\frac{v_p(c)-v_p(a_0)}{n}\right\rceil\leq v_p(d)$. Since
$$v_p(d(\gamma))=\max\left(0,\max_{0\leq j\leq n-1}\left\lceil\frac{v_p(a_n)-v_p(a_j)}{n-j}\right\rceil\right),$$
and $v_p(d(\gamma))\neq0$, this shows that $v_p(d(\gamma))\leq v_p(d)$.

Next if $p\in Y$ satisfies $v_{p}(d)+v_p(e)\leq v_p(c)$, then $v_p(a_{n-1})=v_p(c)-v_p(d)$. In this case $\left\lceil\frac{v_p(c)-v_p(a_{n-1})}{n-(n-1)}\right\rceil=v_p(d)$ and hence $v_p(d)\leq v_p(d(\gamma))$.

Finally suppose $p\in Y$ satisfies $nv_p(d)=v_p(c)$. Since $v_p(a_0)=0$, this means that $\left\lceil\frac{v_p(c)-v_p(a_0)}{n}\right\rceil= v_p(d)$. Therefore $v_p(d)\leq v_p(d(\gamma))$.
\end{proof}

\begin{proof}[Proof of Theorem \ref{classify c,d,e,n}]
Follows from Proposition \ref{c,d,e,n nec cond} and Proposition \ref{c,d,e,n suf cond}.
\end{proof}


\begin{thebibliography}{}
\bibitem{d(gamma)}
S. Arno, M.L. Robinson, F. S. Wheeler,
On Denominators of Algebraic Numbers and Integer Polynomials,
Journal of Number Theory,
Volume 57, Issue 2,
1996,
Pages 292-302.

\bibitem{v_p(d)=k}
M. Ayad, A. Bayad, O. Kihel,
Denominators of algebraic numbers in a number field,
Journal of Number Theory,
Volume 149,
2015,
Pages 1-14.

\bibitem{i gamma 1}
M. Ayad, A. Bayad, O. Kihel.
On fixed divisors of the minimal polynomials over
$\Z$ of algebraic numbers.
Hardy-Ramanujan Journal, Hardy-Ramanujan Society, 2020.

\bibitem{i gamma 2}
M. Ayad, O. Kihel,
Common divisors of values of polynomials and common factors of indices in a number field,
International Journal of Number Theory, 2011, Vol. 07, No. 05, 1173-1194

\bibitem{Z[gamma]cap Q}
P. Drungilas, A. Dubickas, J. Jankauskas.
On relations for rings generated by algebraic numbers and their conjugates.
Annali di Matematica 194, 369–385 (2015).

\bibitem{i gamma 3}
H. T. Engstrom,
On the common index divisors of an algebraic field,
Trans. Amer. Math. Soc. 32, 1930, 223-237.

\bibitem{i gamma 4}
H. Gunji, D. L. McQuillan,
On a class of ideals in an algebraic number field,
Journal of Number Theory,
Volume 2, Issue 2,
1970,
Pages 207-222.

\bibitem{i gamma 5}
H. Gunji, D. L. McQuillan,
On polynomials with integer coefficients,
Journal of Number Theory,
Volume 1, Issue 4,
1969,
Pages 486-493.

\bibitem{i gamma 6}
C.R. MacCluer,
Common divisors of values of polynomials,
Journal of Number Theory,
Volume 3, Issue 1,
1971,
Pages 33-34.

\bibitem{Neukirch 2}
J. Neukirch, A. Schmidt, K. Wingberg (2013).
Chapter VIII. Section 3. In Cohomology of number fields. Springer.






\end{thebibliography}
\end{document}